\documentclass[12pt,a4paper]{amsart}
\usepackage[top=35mm, bottom=30mm, left=30mm, right=30mm]{geometry}
\usepackage{mathptmx}
\usepackage{mathrsfs}
\usepackage{array}

\usepackage{verbatim}
\usepackage{url}
\usepackage[all]{xy}
\usepackage{xcolor}
\usepackage[colorlinks=true,citecolor=blue]{hyperref}

\usepackage{amsmath,amsthm}
\usepackage{amssymb}

\usepackage{enumerate}

\usepackage{graphicx}

\newtheorem{thm}{Theorem}[section]

\newtheorem{lem}[thm]{Lemma}

\theoremstyle{definition}

\numberwithin{equation}{section}

\begin{document}

%%%%% To ease editing, for IMPAN journals add:

\baselineskip=17pt

%%%%%%%%%%%%%%%%

\title{Groups acting distally and minimally on $\mathbb S^2$ and $\mathbb {RP}^2$}

\author{Enhui Shi}

\address[E.H. Shi]{School of Mathematical Sciences, Soochow University, Suzhou 215006, P. R. China}
\email{ehshi@suda.edu.cn}

\author[H.~Xu]{Hui Xu}
\address[H. Xu]{CAS Wu Wen-Tsun Key Laboratory of Mathematics, University of Science and Technology of China, Hefei, Anhui 230026, China}
\email{huixu2734@ustc.edu.cn}

\begin{abstract}
Let $X$ be the $2$-sphere $\mathbb S^2$ or the real projective plane $\mathbb {RP}^2$. We show that if $\Gamma$ is a finitely generated group acting minimally and distally on $X$, then $\Gamma$ contains a nonabelian free subgroup.
\end{abstract}

\keywords{distality, amenable group, free group, minimality, cohomology}
\subjclass[2010]{37B05}

\maketitle

\pagestyle{myheadings} \markboth{E. H. Shi \& Hui Xu }{Distal minimal group actions}

\section{Introduction}

The aim of the note is  continuing the study of the following question:
\medskip

\centerline{\it \ \ \ \ Given a discrete group $G$ and a compact metric space $X$, }
\centerline {\it can $G$ act on $X$ distally and minimally?}
\medskip

\noindent The answer to this question involves the discussions around the algebraic structure of $G$ and the topology of $X$.
Here, we are mainly concerned on the case that $X$ is a closed surface (a compact connected $2$-manifold with no boundary).
There have been several related investigations around this topic.  A remarkable result due to Furstenberg says that if a nontrivial space $X$ admits a
 distal minimal action by a locally compact abelian group, then $X$ cannot be simply connected (see \cite[Theorem 11.1]{Fu} or \cite[Chapter 7-Theorem 16]{Au}).
  In \cite{Shi}, Shi proved further that if a continuum $X$ admits a distal minimal amenable group action, then
  the first {\v C}ech cohomology group $\check H^1(X)$ of $X$ with integer coefficients is nontrivial; in particular, if
  $X$ is a CW-complex, then the fundamental group of $X$ cannot be finite;
  so, the $2$-sphere $\mathbb S^2$ and the real projective plane $\mathbb {RP}^2$ admit no distal minimal actions by amenable groups.
Bron{\v s}te{\v \i}n proved that if $X$ is a connected and locally connected finitely dimensional compact metric space  which admits a distal minimal group action,
then $X$ must be a manifold and the fundamental group $\pi_1(X)$ is virtually nilpotent (\cite{Br}); this implies that if $X$ is a closed surface
except for the sphere $\mathbb S^2$, the real projective plane $\mathbb {RP}^2$, the torus $\mathbb T^2$, and the Klein bottle $\mathbb K^2$,
then it admits no distal minimal actions by any group. Shi showed that no closed surface admits a distal minimal action by $SL(n, \mathbb Z)$ with $n\geq 3$ (\cite{Sh}).

\medskip

The following is the main theorem of this paper. Recall that a group $G$ is a {\it small group} if it contains no free nonabelian subgroups.

\begin{thm}\label{main theorem}
Let $X$ be the $2$-sphere $\mathbb S^2$ or the real projective plane $\mathbb {RP}^2$. If $\Gamma$ is a finitely generated group acting minimally and distally on $X$, then $\Gamma$ contains a nonabelian free subgroup. Equivalently, $X$ admits no distal minimal actions by a small group.
\end{thm}

Here we remark that the class of small groups is strictly larger than that of amenable groups, so this theorem is not implied by the main theorem
in \cite{Shi}; and it is easy to construct distal minimal actions on $\mathbb S^2$ and $\mathbb {RP}^2$ by $\mathbb Z*\mathbb Z$. In addition,
the theorem does not hold when $X$ is either the torus $\mathbb T^2$ or the Klein bottle $\mathbb K^2$, since they admit distal minimal actions by abelian groups (see the appendix).

\medskip
Now we summarize all the known results around the existence of distal minimal group actions on closed surfaces in the following tabular.
\medskip

\begin{center}
\begin{tabular}{|c|c|c|}\hline
\ \ \ \ \   Closed surfaces\ \ \ \ &\ \ \ \ \ \  Existence\ \ \ \ \ \ & Non-existence\\ \hline
$\mathbb{S}^{2}, \mathbb{R P}^{2}$ & $\mathbb{Z}*\mathbb{Z}$ & small groups, $SL(n,\mathbb{Z}) (n\geq 3)$\\ \hline
$\mathbb{T}^{2}, \mathbb{K}^{2} $& $\mathbb{Z}$ & $SL(n,\mathbb{Z}) (n\geq 3)$\\ \hline
Others &  & Any groups\\ \hline
\end{tabular}
\end{center}

\section{Preliminaries}

In this section, we just list the theorems that will be used in the proof of the main theorem without mentioning
the related notions and notations already appeared in \cite{Shi, Sh}.

\begin{thm}\cite[p.98]{Au} \label{homomorphism open}
Let $(X, G, \phi)$ and  $(Y, G, \psi)$ be distal minimal actions, and let $f:X\rightarrow Y$ be a homomorphism.
Then $f$ is open.
\end{thm}

\begin{thm}\cite[p.104]{Au} \label{maximal fator}
Suppose $X$ is not a single point. If $(X, G, \phi)$ is distal minimal, then it has a nontrivial equicontinuous factor.
\end{thm}

\begin{thm}\cite[Theorem 3.17.12]{Br}\label{equicon extension}
Let  $\pi: (X,G, \phi)\rightarrow (Y, G, \psi)$ be a homomorphism between minimal systems. Suppose that $\pi$ is open and $G$ is finitely generated. If $(Y, G,\psi)$ is equicontinuous and there is some  $y\in Y$ such that $f^{-1}(y)$ is of $0$-dimension. Then $(X, G, \phi)$ is also equicontinuous.
\end{thm}

\begin{thm}\cite[p.52]{Au} \label{compact group}
Let $(X, G, \phi)$ be  equicontinuous. Then the closure $\overline{\phi(G)}$ in $C(X,X)$ with respect to
the uniform convergence topology is a compact topological group.
\end{thm}

\begin{thm}\cite{Rees}\label{dimension}
Let $(X, G, \phi)$ and  $(Y, G, \psi)$ be distal minimal actions, and let $f:X\rightarrow Y$ be a homomorphism.
Then for every $y\in Y$, we have ${\rm dim}(Y)+{\rm dim}(f^{-1}(y))={\rm dim}(X)$.
\end{thm}

\begin{thm}\cite[Theorem 3.17.10]{Br}\label{manifold}
Let $(X, G, \phi)$ be a distal minimal system with $X$ being a connected and locally connected compact metric space of
finite dimension, then $X$ is a manifold.
\end{thm}

\begin{thm}\cite[Corollary 4.25]{Kn} \label{compact Lie group}
Let $G$ be a compact Lie group and let $\mathfrak{g}$ be the Lie algebra of $G$. Then
$\mathfrak{g}=Z(\mathfrak{g})\bigoplus [\mathfrak{g}, \mathfrak{g}]$, where $Z(\mathfrak{g})$ is the center of $\mathfrak{g}$ and
$[\mathfrak{g}, \mathfrak{g}]$ is semisimple.
\end{thm}

\begin{thm}\cite[Corollary 1.103]{Kn} \label{abelian compact}
Let $G$ be a compact connected commutative Lie group of dimension $n$. Then $G$
is isomorphic to the $n$-torus $\mathbb T^n$.
\end{thm}

\begin{thm}\cite[Theorem 3.50]{War} \label{center}
Let $G$ be a connected Lie group with Lie algebra $\mathfrak{g}$. Then the center of $G$ is a closed Lie subgroup of $G$ with
Lie algebra the center of $\mathfrak{g}$.
\end{thm}

\begin{thm}\cite[p.65]{MZ} \label{homogeneous space}
Let $X$ be a compact metric space and let $(X, G)$ be an action of group $G$ on $X$. Suppose $G$ is compact. Then for every
$x\in X$, $G/G_x$ is homeomorphic to $Gx$, where $G_x=\{g\in G: gx=x\}$.
\end{thm}

\begin{thm}\cite[p.99]{MZ} \label{small group}
 Let $G$ be a compact group and let $U$ be an open neighborhood of the identity $e$.
Then $U$ contains a closed normal subgroup $H$ of $G$ such that $G/H$ is isomorphic to a Lie group.
\end{thm}

\begin{thm}\cite[p.61]{MZ} \label{induce action} Let $X$ be a compact metric space and let $(X, G)$ be an action of group $G$ on $X$.
 Suppose $G$ is compact and $H$ is a closed normal subgroup of  $G$. Then  $G/H$ can
 act on $X/H$ by letting $gH\cdot H(x)=H(gx)$ for $gH\in G/H$ and $H(x)\in X/H$.
\end{thm}

We use ${\check H}^1(X)$ to denote the first {\v C}ech cohomology group of $X$ with integer coefficients.

\begin{thm}\cite[Corollary 2.15]{Shi} \label{open inverse}
Let $f:X\rightarrow Y$ be an open map from a continuum $X$ onto a continuum $Y$. Then
$f^*:{\check H}^1(Y)\rightarrow {\check H}^1(X)$ is injective.
\end{thm}

\begin{thm}\cite [Theorem 1.3]{BG} \label{free subgroups}
Let $G$ be a connected non-solvable real Lie group of dimension $d$. Then any finitely generated dense subgroup of $G$
contains a dense free subgroup of rank $2d$.
\end{thm}

It is well known that a compact connected Hausdorff space is locally connected metrizable if and only if it is a continuous image of the closed interval $[0,1]$ (\cite[Theorem 8.18]{Nad}). Thus the following result is direct.
\begin{thm}\label{local connect}
Let $X$ be a compact metric space which is connected and locally connected and $Y$ be a Hausdorff space. If
there is a continuous surjection $f:X\rightarrow Y$, then $Y$ is locally connected and metrizable.
\end{thm}

\section{Proof of the main theorem}

\begin{lem}\label{distal imply equicon}
Let $X$ be the $2$-sphere $\mathbb S^2$ or the real projective plane $\mathbb {RP}^2$. Let $\Gamma$ be a finitely generated group and $\phi:\Gamma\rightarrow {\rm Homeo}(X)$ be a distal minimal action on $X$. Then  $(X, \Gamma, \phi)$ is equicontinuous.
\end{lem}

\begin{proof}
Assume to the contrary that $\phi$ is not equicontinuous. From Theorem \ref{maximal fator}, we let $(Y, \Gamma, \psi)$ be the
maximal equicontinuous factor of $(X, \Gamma, \phi)$ with a factor map $\pi$. Then $Y$ is a compact manifold of
dimension $\leq 2$ by Theorem \ref{dimension}, Theorem \ref{manifold}, and Theorem \ref{local connect}. If ${\rm dim}(Y)=2$, then it follows from
Theorem \ref{dimension} that for every $x\in X$, ${\rm dim}(\pi^{-1}(x))=0$. This together with Theorem \ref{equicon extension} and Theorem \ref{homomorphism open}
implies that  $(X, \Gamma, \phi)$ is equicontinuous, which contradicts the assumption. So, we
may assume that ${\rm dim}(Y)=1$; this means that $Y$ is the circle $\mathbb S^1$. Since $\pi$ is open and ${\check H}^1(\mathbb S^1)$
is the integer group, from Theorem \ref{homomorphism open} and Theorem \ref{open inverse}, we have that ${\check H}^1(X)$ is infinite, which is a contradiction.  \end{proof}

\begin{lem}\label{semisimple}
Let $G$ be a connected compact Lie group acting faithfully and transitively on a closed surface $X$  with finite fundamental group. Then $G$ is semisimple.
\end{lem}

\begin{proof}
Assume to the contrary that $G$ is not semisimple. Then by Theorem \ref{compact Lie group}, Theorem \ref{abelian compact}, and Theorem \ref{center},
 the connected component $Z(G)_0$ of the center  of $G$ is isomorphic
to some torus $\mathbb T^n$ with $n>0$. Set $K=Z(G)_0$. For $x\in X$, let ${\rm Stab}(x):=\{k\in K: kx=x\}$ be the stabilizer of $x$ in $K$.
Then from Theorem \ref{homogeneous space}, $Kx$ is homeomorphic to $K/{\rm Stab}(x)$ which is also a torus. Thus $Kx$ is either a point or a circle. If for every
$x\in X$, $Kx$ is a circle, then similar to the arguments in Lemma \ref{distal imply equicon}, $X/K$ is a circle and the {\v C}ech cohomology group ${\check H}^1(X)$ with integer coefficients is infinite. This is a contradiction.
So, there is some $x_0\in X$ with $Kx_0=x_0$. Since $K$ is in the center of $G$, we have $Kgx_0=gKx_0=gx_0$ for every $g\in G$.
Noting that $Gx_0=X$, the action of $K$ on $X$ is trivial, which contradicts the faithfulness of the action. So, $G$ is semisimple.
\end{proof}

\begin{proof}[Proof of the main theorem] Let $\phi:\Gamma\rightarrow {\rm Homeo}(X)$ be the
distal  minimal action. From Lemma \ref{distal imply equicon}, we see that the $(X, \Gamma, \phi)$ is equicontinuous.
Let $K$ be the closure of $\phi(\Gamma)$ with respect to the uniform topology on ${\rm Homeo}(X)$. It follows from Theorem \ref{compact group}
that $K$ is a compact metric group acting transitively on $X$. By Theorem \ref{small group}, we can take a closed normal subgroup $N$ of $K$
such that $K/N$ is a Lie group and $X/N$ is not a single point. Then  it canonically induces an action of $K$ on $X/N$ and the natural quotient map $X\rightarrow X/N$ is a equicontinuous extension. Thus it follows from Theorem \ref{dimension}, Theorem \ref{manifold}, and Theorem \ref{local connect} that $X/N$ is a manifold of dimension $\leq 2$.
Similar to the arguments in Lemma \ref{distal imply equicon}, we have $\dim(X/N)=2$.

Set $p:\mathbb S^2\rightarrow X$ be a covering and $q: X\rightarrow X/N$ be the quotient map. Let
$\pi:Y\rightarrow X/N$ be the universal covering and $\widetilde{qp}:\mathbb S^2\rightarrow Y$ be the
lifting of $qp$. Since $\pi$ is open and $p$ and $q$ are local homeomorphisms, we see that
$\widetilde{qp}$ is open. Thus $\widetilde{qp}(\mathbb S^2)$ is open and closed in $Y$. Thus
$Y=\widetilde{qp}(\mathbb S^2)$ by the connectedness. So $Y$ is compact (homeomorphic to $\mathbb S^2$). Thus
$\pi$ is a finite cover and then the fundamental group of $X/N$ is finite.

Now consider the natural action of $K/N$ on $X/N$ (see Theorem \ref{induce action}). Since the connected component $(K/N)_0$ has finite
index in $K/N$, the $(K/N)_0$ action on $X/N$ is still transitive and $\phi(\Gamma)N\cap (K/N)_0$ is dense in $(K/N)_0$.
Applying Lemma \ref{semisimple}, we see that a quotient group of $(K/N)_0$ is semisimple. This together with
Theorem \ref{free subgroups} implies the existence of free nonabelian subgroups in $\Gamma$.
\end{proof}

\section{appendix}
In this appendix, we show that there is a minimal distal  homeomorphism on the Klein bottle.

Let the torus $\mathbb{T}^{2}$ be $\mathbb{R}^{2}/\mathbb{Z}^{2}$. Then there is a $\mathbb{Z}_{2}$ action on $\mathbb{T}^{2}$ by
\[ h(x,y)=(x+\frac{1}{2}, 1-y) \mod \mathbb{Z}^{2},\]
where $h$ is the nonidentity element in $\mathbb{Z}_{2}$. It is easy to see that this action is free and properly discontinuous and the quotient space $\mathbb{T}^{2}/\mathbb{Z}_2$ is the Klein bottle $\mathbb{K}^{2}$.

\medskip
Now for a homeomorphism $T$ of $\mathbb{T}^{2}$,  if it commutes with $h$, i.e., $Th=hT$, then it induces a homeomorphism $\widetilde{T}$ of $\mathbb{K}^{2}$.

\medskip
 Let $\alpha$ be an irrational number and $\phi: \mathbb{T}\rightarrow\mathbb{T}$ be a continuous mapping. Further, define a  homeomorphism $T: \mathbb{T}^{2}\rightarrow \mathbb{T}^{2}$  by
\[ T(x,y)=(x+\alpha, y+\phi(x))\mod\mathbb{Z}^{2}.\]
The system $(\mathbb{T}^{2}, T)$ is a skew product system and it is well known that this system is distal, since it is a group extension of a minimal equicontinuous system. The following theorem characterizes the minimality of such skew product system.
\begin{thm}\cite[Chapter 5, Theorem 10]{Au}\label{char of min}
The above defined system $(\mathbb{T}^{2}, T)$ is minimal if and only if for each $k\in \mathbb{Z}\setminus\{0\}$, there is no continuous function $f:\mathbb{T}\rightarrow \mathbb{T}$ such that $f(x+\alpha)=f(x)+k\phi(x)$ for each $x\in\mathbb{T}$.
\end{thm}

Now we take $\phi: \mathbb{T}\rightarrow \mathbb{T}$ to be $\phi(x)=1-|1-2x|$ for $x\in [0,1)$. By comparing the Fourier coefficients, there is no continuous function $f:\mathbb{T}\rightarrow \mathbb{T}$ such that $f(x+\alpha)=f(x)+k\phi(x)$  for each $k\in\mathbb{Z}\setminus\{0\}$. Thus the system $(\mathbb{T}^{2}, T)$ is minimal by Theorem \ref{char of min}. It is straightforward to calculate that  for each $(x,y)\in\mathbb{T}^{2}$, $Th(x,y)=(x+\alpha+\frac{1}{2}, 1-y+\phi(x+\frac{1}{2}))$ and $hT(x,y)=(x+\alpha+\frac{1}{2}, 1-y-\phi(x))$. Note that
\[\phi(x+\frac{1}{2})=\begin{cases} 1-2x, & x\in[0,1/2)\\ 2x-1, & x\in [1/2,1)\end{cases} \text{ and } -\phi(x)=\begin{cases} -2x, & x\in[0,1/2)\\ 2x-2, & x\in [1/2,1)\end{cases}. \]
Therefore, it follows that $T$ commutes with $h$ and thus the induced homeomorphism $\widetilde{T}$ on $\mathbb{K}^{2}$ is also minimal and distal.

%\newpage

\end{document}